\documentclass{amsart}
\usepackage{hyperref}
\usepackage{mathrsfs}
\usepackage{amsmath}
\usepackage{amssymb}
\usepackage[all]{xy}
\usepackage{latexsym}
\usepackage[utf8]{inputenc}
\title{On the $\bpn$-cohomology of elementary abelian $p$-groups}
\author[Geoffrey Powell]{Geoffrey Powell}
\address{Laboratoire Angevin de Recherche en Mathématiques, UMR 6093, 
  Faculté des Sciences, Université d'Angers, 
2 Boulevard Lavoisier,
49045 Angers, France}
\email{Geoffrey.Powell@math.cnrs.fr}
\keywords{Brown-Peterson theory; Johnson-Wilson theories; Milnor primitive; elementary abelian group; formal group}
\subjclass[2000]{55N20; 55N22; 20J06}
\date{}
\thanks{}

\newtheorem{thm}{Theorem}[section]
\newtheorem{prop}[thm]{Proposition}
\newtheorem{cor}[thm]{Corollary}
\newtheorem{lem}[thm]{Lemma}

\theoremstyle{definition}
\newtheorem{defn}[thm]{Definition}

\theoremstyle{remark}
\newtheorem{rem}[thm]{Remark}
\newtheorem{nota}[thm]{Notation}

\newcommand{\double}{{\mathbf{\Psi}}}
\newcommand{\hq}{\mathscr{H}}
\newcommand{\vtor}[1][n]{\mathbf{tors}_{v_{#1}}}
\newcommand{\ploc}[1][p]{\zed_{(#1)}}
\newcommand{\bpn}[1][n]{BP\langle #1 \rangle}
\newcommand{\Strunc}[2]{\overline{S^{#1}_{#2}}}
\newcommand{\cali}{\mathscr{I}}

\newcommand{\nat}{\mathbb{N}}
\newcommand{\zed}{\mathbb{Z}}
\newcommand{\field}{\mathbb{F}}

\renewcommand{\phi}{\varphi}
\renewcommand{\epsilon}{\varepsilon}
\begin{document}

\begin{abstract}
The structure of the $\bpn$-cohomology of elementary abelian $p$-groups is
studied, obtaining a presentation expressed in terms of $BP$-cohomology and
mod-$p$ singular cohomology, using the Milnor derivations. 

The arguments are based on a result on multi-Koszul complexes which is related to 
Margolis's criterion for freeness of a graded module over an exterior algebra. 
\end{abstract}
\maketitle

\section{Introduction}
\label{sect:intro}

Understanding the generalized group cohomology of elementary abelian $p$-groups
for a cohomology theory $E^* (-)$ is of interest both as a first step towards
the study of generalized group cohomology, inspired in part by
the results of Quillen for singular cohomology, and also since Lannes' theory
\cite{Lannes} implies that it yields information on the $p$-local homotopy type
of the spaces of the $\Omega$-spectrum representing $E$. 

In studying the spectra of interest in chromatic homotopy theory, it is natural
to commence by the complex oriented theories. Here the state of knowledge is 
incomplete once one moves outside the cases admitting descriptions 
as formal schemes (see \cite{HKR}) or the classical cases corresponding 
to singular cohomology or the periodic Morava $K$-theories.

The universal example, complex cobordism $MU$, is of interest. For elementary 
abelian $p$-groups,  one can 
reduce to   Brown-Peterson theory, $BP$; this corresponds to working 
$p$-locally, hence 
restricting to $p$-typical formal 
 group laws. Landweber showed that $BP^* (BV)$, for $V$ an
elementary abelian $p$-group, can be described  in terms of the formal group
structure (the situation for Brown-Peterson homology is much more
complicated \cite{JW85,JWY94}).

Wilson \cite{WilI,WilII} introduced and studied  the theories $\bpn$, for $n \in
\nat$, which
interpolate between $BP=\bpn[\infty] $ and the mod $p$ Eilenberg-MacLane 
spectrum 
$H\field_p= \bpn[-1]$. These  provide a first step towards other theories of
significant interest in chromatic homotopy theory; moreover, they are 
important in understanding the $BP$-cohomology of Eilenberg-MacLane spaces (cf.
\cite{RWY}).

The cases $\bpn[-1] = H \field_p$,  $\bpn[0] = H \zed_{(p)}$ and $\bpn[1]$ are 
understood ($\bpn[1]$
identifies with the Adams summand of $p$-local connective complex $K$-theory).
Hitherto, for $n >1$,  results on $\bpn^* (BV)$ have concentrated on  low degree
or small rank behaviour; for example, Strickland \cite{Str} gave an analysis of
the first (in terms of the rank) occurrence of $v_n$-torsion in $\bpn^* (BV)$,
exhibiting a relationship between 
formal group theory and the action of the Milnor derivations on $H\field_p^* 
(BV)$. 

This paper shows that this is the tip of the iceberg: the Milnor derivations
explain all the $v_n$-torsion, without restriction on the rank of $V$. 
 The structure of $\bpn^* (BV)$ is determined in terms of the 
contribution from formal groups  obtained from $BP^* (BV)$ by base
change, and from mod-$p$ cohomology  $H\field_p^* (BV)$, considered as a module
over $\Lambda (Q_0, \ldots , Q_n)$. Namely, there is a short exact sequence 
\[
 0\rightarrow
L_n \hookrightarrow  \big(\bpn[n]^* \otimes_{BP^*} BP^* (BV)\big) \oplus
\vtor[n] \twoheadrightarrow \bpn[n]^*
(BV) \rightarrow 0,
\]
where the $v_n$-torsion $\vtor[n] \subset \bpn[n]^* (BV)$ is a trivial 
$\bpn[n]^*$-module, which is 
isomorphic to the image $\mathrm{Im} (Q_0
\ldots Q_n) \subset H\field_p^* (BV)$ of the iterated Milnor operation, and the 
kernel $L_n$ is identified explicitly 
 (see Theorem \ref{thm:eltab}).

This result is derived from the general result, Theorem \ref{thm:general}, for which the 
key input is the behaviour of the quotients (for $n \in \nat$):
 \[
  \hq^* (X,n) : = \big\{\bigcap _{i=0} ^n \mathrm{Ker} (Q_i) \big\} /
\mathrm{Im} (Q_0 \ldots Q_{n}).
 \]
 associated to the mod $p$-cohomology of a space $X$. The fundamental property is that 
the Thom reduction from $BP$ to mod-$p$ cohomology induces a surjection onto $\hq^* (X,n)$.

 The proof of this for the case $X =BV$ is a modification of Margolis's
criterion  \cite{margolis}  for a module over the exterior algebra $\Lambda
(Q_0, \ldots , Q_n)$ on the Milnor derivations $Q_i$ to be free; this
establishes a fundamental  property of the structure of  $H\field_p^* (BV)$ (see Theorem
\ref{thm:multikoszul}).

The argument can be generalized to the study of any $MU$-module spectrum which
is
constructed from $BP$ by forming the quotient by a cofinite subset of a 
suitable 
set 
of algebra generators $\{ v_i |
i \in \nat \}$ for $BP_*$ (where $v_0 =p$). For instance, the methods recover 
the author's results on
connective complex $K$-theory \cite{powell}; moreover, they also apply to
connective Morava $K$-theories, adding a useful perspective on existing results,
such as Kuhn's study of the periodic theory \cite{kuhn87} and the results of
Wilson on the Hopf ring of periodic Morava $K$-theory \cite{wil84}, and Hara,
on the Hopf ring of the connective theory \cite{hara}. 
 Similarly, the methods extend to the study of integral version of connective
Morava $K$-theory, generalizing the results for connective complex $K$-theory.
 For simplicity of exposition, these applications are not treated in the current
paper;  however, the main input is provided by Proposition
\ref{prop:multi-koszul}, which is proved in full generality.

 \bigskip
 {\bf Organization of the paper:}
 Section \ref{sect:prelim} provides background and Section
\ref{sect:thom_reduct} introduces the subquotient which bounds the indeterminacy
of the Thom reduction map in terms of the action
 of the Milnor primitives. Section \ref{sect:general} proves technical results
which control injectivity and surjectivity of certain reduction maps. The
fulcrum is Section \ref{sect:noeth}, which shows 
 how the $v_n$-torsion can be controlled in odd degrees under appropriate
hypotheses; Section \ref{sect:trivial_torsion_thm} exhibits the ramifications to
the full  $\bpn$-cohomology. Finally, in Section \ref{sect:eltab}, these
techniques are applied to the case of elementary abelian $p$-groups, proving the
Margolis-type vanishing result, which provides the necessary input.

\section{Preliminaries}
\label{sect:prelim}


\subsection{Torsion theories}
This section fixes notation and recalls a standard result on the relation
between 
torsion submodules and annihilator submodules.

Let $R$ be a commutative ring and $R[v]$ the polynomial algebra on $v$. For $M$
 an $R[v]$-module, the $v$-torsion submodule $\mathbf{tors}_vM$ is the set of
$v$-torsion elements $\{m \in M | \exists t \ v^t m=0 \}$ and $\mathrm{Ker}_vM$
is the kernel of multiplication by $v$, $M \stackrel{v}{\rightarrow} M $, so
that $\mathrm{Ker}_v M \cong \mathrm{Tor}^{R[v]} _1 (R, M)$  and 
$\mathrm{Ker}_v M \subset \mathbf{tors}_v M $. The $v$-cotorsion
$\mathbf{cotors}_v M$ is the quotient $M / \mathbf{tors}_v M$,
so that there is a natural short exact sequence 
\[
 0 
 \rightarrow 
 \mathbf{tors}_v M
 \rightarrow 
 M
 \rightarrow 
 \mathbf{cotors}_v M 
 \rightarrow 
 0.
\]
This is a standard example of a hereditary torsion theory.

The proof of the following is straightforward.

\begin{lem}
\label{lem:ker-tor}
For $M$ an $R[v]$-module, the following conditions are equivalent:
\begin{enumerate}
 \item 
 $\mathrm{Ker}_v M= \mathbf{tors}_v M$;
 \item
 $vM \cap \mathrm{Ker}_v M=0$;
 \item
 the projection $M \twoheadrightarrow M/ vM$ induces a monomorphism 
$\mathrm{Ker}_v M \hookrightarrow M/vM$.
\end{enumerate}
If these conditions are satisfied, there is a  short exact
sequence 
\[
 0 
 \rightarrow 
 \mathrm{Ker}_v M
 \rightarrow M/vM
 \rightarrow 
 (\mathbf{cotors}_v M)/ v
 \rightarrow 
 0.
\]
\end{lem}

\begin{rem}
 In the application, rings and modules are graded; as usual,  the appropriate
commutativity condition is graded commutativity (with Koszul signs). 
 However, where this intervenes, the rings are concentrated in even degrees, so
signs do not appear.
 \end{rem}

\subsection{The Wilson theories $\bpn$}

Fix a  prime $p$ and consider the Brown-Peterson spectrum $BP$ and the
associated Wilson spectra $\bpn$ (cf. \cite{WilII,Str,Tam}), equipped with
the reduction maps 
$$BP \stackrel{\rho_n}{\rightarrow} \bpn \stackrel{\rho^n_{n-1} }{\rightarrow}
\bpn[n-1],$$ 
which can be constructed in the category of $MU$-modules. The $\bpn$ can be
taken to be commutative $MU$-ring
spectra so that the reduction maps are morphisms of ring spectra \cite[Section
3]{Str}.  The coefficient rings are 
$BP_* \cong \ploc[p] [v_i  |i \geq 0]$, 
$\bpn_* \cong BP_*/ (v_i | i >n)$,   where  $|v_n| = 2 (p^n -1)$ and  $v_0 =p$,
by convention; thus $\bpn_* \cong  \ploc[p][v_1, \ldots, v_n]$ for $n\geq 1$.
In particular $\bpn[-1] = H \field_p$ and  $\bpn[0] = H \ploc$ are
Eilenberg-MacLane spectra.

Multiplication by $v_n$  fits into the cofibre sequence which defines $q_n$:
\begin{eqnarray}
\label{eqn:cofib_BPn}
 \Sigma^{|v_n|} \bpn \stackrel{v_n} {\rightarrow} 
 \bpn 
 \stackrel{\rho^n_{n-1}} {\rightarrow}
 \bpn[n-1] 
 \stackrel{q_n}{\rightarrow}
 \Sigma^{|v_n|+1} \bpn.
\end{eqnarray}
The following  is clear:

\begin{lem}
 \label{lem:q_vn-torsion}
 For $X$ a spectrum and $n \in \nat$, $q_n$ induces a map 
 \[
  \bpn[n-1] ^* (X) \stackrel{q_n}{\longrightarrow} \mathrm{Ker}(v_n) ^{*+
|v_n|+1} \subset \bpn ^{*+ |v_n|+1} (X).
 \]
\end{lem}

The composite $\rho_n q_n : \bpn[n-1] \rightarrow \Sigma^{|v_n|+1} \bpn[n-1]$ is
a derivation (cf.  \cite[Section 3]{Str}). More generally, as in  {\em loc.
cit.}, one considers the derivation 
induced for $MU$-modules by the derivation $MU/v_n  \rightarrow \Sigma^{|v_n|+1}
MU/v_{n}$, which provides compatibility; 
 the operation on $H\field_p$ coincides  with the Milnor derivation $Q_n$ (up to
sign), by \cite[Proposition 3.1]{Str}.

This compatibility  implies the following (cf. 
\cite[Proposition 4-4]{Tam}):

\begin{lem}
 \label{lem:Q_compatible}
 For $n \in \nat$, the following diagram commutes
 \[
  \xymatrix{
  \bpn
  \ar[r]^(.35){q_{n+1}} 
  \ar[d]_{\rho^n_{-1}}
  &
  \Sigma^{|Q_{n+1}|} 
  \bpn[n+1]
  \ar[d]^{\rho^{n+1}_{-1}}
  \\
  H\field_p
  \ar[r]_(.4){\pm Q_{n+1}}
  &
  \Sigma^{|Q_{n+1}|}H\field_p.
  }
 \]
 Hence, (up to possible sign) the composite $Q_n \ldots Q_0 : H \field_p
\rightarrow 
 \Sigma^{\sum |Q_i| } H\field_p$ factors across $\rho^n _{-1}$ as 
 \[
  \xymatrix{
  H\field_p 
  \ar[rr]^(.4){q_n\ldots q_0} 
  &&
  \Sigma^{\sum |Q_i|} \bpn 
  \ar[r]^{\rho^n _{-1}}
  &
  \Sigma^{\sum |Q_i| } H\field_p.
  }
 \]
\end{lem}

When considering the $\bpn$-cohomology of a space, the following  can be 
applied.

\begin{prop}
 \label{prop:BPodd0_v_n-torsion}
 For $X$ a space such that $BP ^{\mathrm{odd}} (X)=0$ and $n \in \nat$,
$\bpn^{\mathrm{odd}} (X)$
is $v_i$-torsion for $0\leq i \leq n$. 
\end{prop}

\begin{proof}
 \cite[Corollary 5.6]{WilII} shows that the reduction map $(\rho_{n})^t : 
BP^t(X) \rightarrow \bpn^t
(X)$ is surjective for $t\leq 2 \big( \frac{p^n-1}{p-1} \big)$, hence $ \bpn^t
(X)$ is zero in this range. The result  is a straightforward consequence.
\end{proof}

The condition $\bpn[n-1]^{\mathrm{odd}}(X)=0$ arises naturally at the start of the 
inductive arguments; the following  observation records its immediate 
ramifications. 

\begin{prop}
 \label{prop:v_n-torsion}
 For $X$ a spectrum and $0<n \in \nat$ such that $\bpn[n-1]^{\mathrm{odd}}(X)=0$, the
following properties hold:
 \begin{enumerate}
  \item 
  $\vtor[n]^{\mathrm{even}}=0$, where $\vtor[n]\subset \bpn[n]^* (X)$;
  \item
  if $\bpn[n]^{\mathrm{odd}}(X)$ contains no $v_{n}$-divisible elements,
then $$\bpn[n]^{\mathrm{odd}}(X) =0.$$
 \end{enumerate}
 \end{prop}

\begin{proof}
 The result follows from the long exact sequence associated to the cofibre 
sequence (\ref{eqn:cofib_BPn}). For example, 
the hypothesis $\bpn[n-1]^{\mathrm{odd}}(X)=0$ implies that 
 any element of $\bpn[n]^{\mathrm{odd}}(X)$ is the image of an odd degree 
element under multiplication 
 by $v_{n}$; repeating the argument, any such element is (infinitely) 
$v_{n}$-divisible.
 \end{proof}

\section{The image of the Thom reduction}
\label{sect:thom_reduct}

The image in cohomology of the Thom reduction map $BP \rightarrow H \zed_{(p)}$
is of
significant interest in general (see
\cite{tam_BPimage}, for example);  
here we consider the image of 
$ 
\rho^n_{-1} : \bpn \rightarrow H\field_p
$ 
and its relation with the action of the Milnor derivations $Q_i$ on mod-$p$
cohomology.

The following is well-known; a proof is included for the convenience of the
reader.

\begin{prop}
\label{prop:indeterminacy}
For $X$ a spectrum and $n \in \nat$, the reduction map $\rho^n_{-1}$ induces a
map of $\bpn^*$-modules: $\rho^n_{-1} :
\bpn^* (X) \rightarrow H\field_p^* (X)$ such that 
\[
 \mathrm{Im}(Q_0 \ldots Q_{n})
 \subset 
 \mathrm{Image} (\rho^n_{-1})
 \subset \bigcap _{i=0} ^n \mathrm{Ker} (Q_i).
\]
\end{prop}

\begin{proof}
The inclusion $\mathrm{Im}(Q_0 \ldots Q_{n})
 \subset 
 \mathrm{Image} (\rho^n_{-1})$ is a consequence of the factorization of $Q_0
\ldots Q_{n}$ across $\rho^n_{-1}$, given by Lemma \ref{lem:Q_compatible}.

For the upper bound, since $\rho^n _{-1} = \rho^{n-1}_{-1} \rho^n_{n-1}$, it
suffices to show that 
$ \mathrm{Image} (\rho^n_{-1})\subset \mathrm{Ker}(Q_n)$; this follows from
the commutative
diagram
\[
 \xymatrix{
\bpn^* (X) 
\ar[r]_{\rho^n_{n-1}}
\ar[rd]_{\rho^n_{-1}}
\ar@/^1pc/[rr]^0
&
\bpn[n-1]^* (X)
\ar[d]^{\rho^{n-1}_{-1}}
\ar[r]_{q_n}
&
\bpn^{*+|Q_n|}(X)
\ar[d]^{\rho^n_{-1}}
\\
&
H\field_p^* (X) 
\ar[r]_{\pm Q_n}
&
H\field_p^{*+|Q_n|} (X),
}
\]
where the commutative square is provided by  Lemma \ref{lem:Q_compatible}.
 \end{proof}

\begin{nota}
 For $X$ a spectrum and $n \in \nat$, let $\hq ^*(X,n)$ denote the graded
subquotient of $H\field_p^* (X)$ 
 \[
  \hq^* (X,n) : = \big\{\bigcap _{i=0} ^n \mathrm{Ker} (Q_i) \big\} /
\mathrm{Im} (Q_0 \ldots Q_{n}).
 \]
\end{nota}

 \begin{rem}
 Proposition \ref{prop:indeterminacy} shows that $\hq^* (X,n)$ bounds the
indeterminacy
of the image of $\rho^n_{-1}$. In particular, if $\hq^t
(X,n)=0$, then $\mathrm{Image} (\rho^n_{-1})^t = \mathrm{Im}(Q_0 \ldots
Q_{n})^t$.
 \end{rem}

\begin{rem}
\label{rem:free_trivial_hq}
For $M$  a graded module over  the exterior algebra $\Lambda (Q_0, \ldots,
Q_n)$, 
$\bigcap _{i=0}^n \mathrm{Ker} (Q_i) \subset M$ identifies with the socle
$\mathrm{soc}(M)$ of $M$.
If
$M$ is bounded below and of finite
type, it can be written as $M \cong F \oplus \overline{M}$, where $F$ is a free
$\Lambda (Q_i| 0 \leq i \leq n)$-module and $\overline{M}$ contains no free
sub-module (see \cite{margolis}, for example).
 The inclusion $\mathrm{Im} (Q_0\ldots Q_n) \subset \bigcap _{i=0}^n
\mathrm{Ker} (Q_i)$ corresponds to the
 inclusion $\mathrm{soc} (F) \hookrightarrow \mathrm{soc}(M)$ and the quotient
 identifies with $\mathrm{soc} (\overline{M})$. 

Hence, $\hq^* (X,n)$ gives a measure of the failure of $H\field_p^* (X)$ to be
free as an $\Lambda (Q_0, \ldots , Q_n)$-module, when $X$ is a connective
spectrum with cohomology of finite type.

At the opposite extreme, if $Q_0, \ldots , Q_n$ act trivially upon  $ 
H\field_p^* (X)$ (for example, if the latter 
is concentrated in even degrees), then there is an identification  $\hq^* (X,n) 
\cong H\field_p^* (X)$.
\end{rem}

By Proposition \ref{prop:indeterminacy}, $\rho^n_{-1}$
maps to $\bigcap
_{i=0}^{n}\mathrm{Ker}(Q_i) \subset H\field_p^* (X)$, hence induces a map to
$\hq^* (X,n)$.

\begin{cor}
\label{cor:surjectivity}
For $X$  a spectrum and $n\in \nat$, 
 $\rho^n_{-1} : \bpn^* (X) \rightarrow H\field_p^* (X)$ surjects to
$\bigcap
_{i=0}^{n}\mathrm{Ker}(Q_i) $ if and only if the induced map 
$\bpn^* (X) \rightarrow \hq^* (X,n)$ is surjective.
\end{cor}

\begin{proof}
A straightforward consequence of Proposition \ref{prop:indeterminacy}.
\end{proof}

\begin{rem}
Surjectivity to  $\bigcap
_{i=0}^{n}\mathrm{Ker}(Q_i) $
is a natural condition; for $n=1$ it arises in the work of Kane 
\cite{kane} on finite $H$-spaces via connective $K$-theory.
\end{rem}

When $X$ is a space, further information can be obtained by exploiting 
multiplicative structure. (Henceforth, cohomology is taken to be reduced, so a 
disjoint 
basepoint is required.)

\begin{prop}
 \label{prop:mult_quotient}
 For $X$ a space and $n \in \nat$,  
 \begin{enumerate}
  \item 
  the cup product on $H\field_p^* (X_+) $ induces a graded commutative algebra
structure on $\hq^*(X_+, n)$; 
  \item
  the reduction map $\bpn^* (X_+) \rightarrow \hq^* (X_+,n)$ is a morphism of
$\bpn^*$-algebras.
 \end{enumerate}
\end{prop}

\begin{proof}
 The first statement is an immediate consequence of the fact that the operations
$Q_i$ are derivations and the second is a formal consequence of the construction
of the reduction. 
\end{proof}

\section{Injectivity and surjectivity for generalized reduction maps}
\label{sect:general}

Fix $n\in \nat$;  for a spectrum $X$, 
$\rho^{n+1}_n : \bpn[n+1] \rightarrow \bpn$ induces a morphism of $\bpn[n+1]_*$-modules
\[
\bpn[n+1]^* (X)
\stackrel{\rho^{n+1}_n}{\rightarrow}
\bpn^* (X), 
\]
which is not surjective in general. Similarly one can consider the reduction 
\[
BP^* (X)
\stackrel{\rho_n}{\rightarrow}
\bpn^* (X). 
\]
Wilson's result, \cite[Corollary 5.6]{WilII}, gives surjectivity in low
degrees, for $X$ a suspension spectrum. 

General criteria for injectivity and surjectivity are introduced in this
section.

\subsection{Surjecting to $\bpn$-cohomology}

The short exact sequence 
\begin{eqnarray}
 \label{eqn:ses_fund}
 0 
 \rightarrow
 \bpn[n+1] ^* (X) /v_{n+1}
 \rightarrow
\bpn^* (X)
\rightarrow 
\mathrm{Ker}(v_{n+1})^{*+|Q_{n+1}|}
\rightarrow
0
\end{eqnarray}
is induced by the cofibre sequence 
$\bpn[n+1]
\stackrel{\rho^{n+1}_n } {\rightarrow}
\bpn
\stackrel{q_{n+1}}{\rightarrow} 
\Sigma^{|Q_{n+1}|} 
\bpn[n+1].
$

\begin{rem}
Identifying $\mathrm{Ker}(v_{n+1}) $ as $\mathrm{Tor}_1^{\zed_{(p)}[v_{n+1}]} (
\zed_{(p)} ,\bpn[n+1]^* (X))$, the sequence (\ref{eqn:ses_fund}) can be viewed
as a universal coefficient short exact sequence; cf. \cite[Proposition
5.7]{JW73}, where homology is considered. 
\end{rem}

By restriction  to  $\vtor \subset \bpn^*
(X)$, $q_{n+1}$ gives a natural map 
\[
 \kappa_n : \vtor \rightarrow \Sigma^{|Q_{n+1}|} \mathrm{Ker}(v_{n+1});
\]
and the inclusion $\vtor \subset \bpn^* (X)$ together with 
$\bpn[n+1]^* (X) \stackrel{\rho^{n+1}_n}{\rightarrow} \bpn^* (X)$ induce
\[
\sigma_n : 
 \bpn[n+1]^* (X) \oplus \vtor \rightarrow
 \bpn^* (X).
\]
Similarly, write 
$$\tilde{\sigma}_n : BP^* (X) \oplus \vtor
\rightarrow
 \bpn^* (X)$$
 for the map obtained by replacing $\rho^{n+1}_n$ with $\rho_n$.

\begin{lem}
\label{lem:equivalent_surjectivity}
 For $X$ a spectrum, the following conditions are equivalent:
 \begin{enumerate}
  \item 
  $\sigma_n : \bpn[n+1]^* (X) \oplus
\vtor \rightarrow
 \bpn^* (X)$ is surjective;
  \item
   $\kappa_n : \vtor \rightarrow
\Sigma^{|Q_{n+1}|} \mathrm{Ker}(v_{n+1})$ is surjective;
\item
 $
  \bpn[n+1]^* (X) 
  \rightarrow 
  \mathbf{cotors}_{v_{n}} \bpn^* (X), 
 $
 induced by $\rho^{n+1}_n$, is surjective.
 \end{enumerate}
\end{lem}

\begin{proof}
 Straightforward.
\end{proof}

The following result illustrates how the identification of $ \bpn[n+1]^\mathrm{odd}(X)$ leads to 
 a criterion for the surjectivity of $\sigma_n$; this is a warm-up for the proof of Theorem \ref{thm:general}.

\begin{prop}
\label{prop:sigma_surjectivity}
Let $X$ be a spectrum  such that $\bpn^\mathrm{odd} (X) =
\vtor^\mathrm{odd}$ and 
  $\rho^{n+1}_{-1}$ induces an isomorphism
 \[
  \bpn[n+1]^\mathrm{odd}(X) \stackrel{\cong}{\rightarrow} \mathrm{Im} (Q_0\ldots
Q_{n+1})^\mathrm{odd} 
  \subset H\field_p^\mathrm{odd} (X).
 \]
Then $\sigma_n : \bpn[n+1]^* (X) \oplus \vtor \twoheadrightarrow \bpn^* (X)$ is
surjective. 
\end{prop}

\begin{proof}
 Since $\vtor^\mathrm{odd} \cong \bpn^\mathrm{odd} (X)$ by hypothesis, it
suffices to show that 
 $\sigma_n$ surjects in even degree. By Lemma \ref{lem:equivalent_surjectivity},
it suffices to show that 
 \[
\kappa_n :  \vtor^{2s} \rightarrow \mathrm{Ker}(v_{n+1}) ^{2s + |Q_{n+1}|}
 \]
is surjective, for all $s\in \zed$. The hypothesis on $\bpn[n+1]^\mathrm{odd}
(X)$ implies that 
$\mathrm{Ker}(v_{n+1}) ^\mathrm{odd} = \bpn[n+1]^\mathrm{odd} (X)$, 
which embeds as
$\mathrm{Im} (Q_0\ldots Q_{n+1})^\mathrm{odd}$ in $H\field_p^\mathrm{odd}(X)$.
 
Lemma \ref{lem:Q_compatible} shows that $Q_0\ldots Q_{n+1}$
factors across $q_{n+1}\ldots q_0$ and, hence, across $q_n\ldots q_0 : H\field_p
\rightarrow \Sigma^{\sum |Q_i|}
\bpn$, which maps to $ \mathrm{Ker}(v_n) \subset \vtor \subset \bpn^* (X)$ in
cohomology, by Lemma \ref{lem:q_vn-torsion}; since $\kappa_n$
is induced by $q_{n+1}$, Lemma \ref{lem:Q_compatible}  implies  surjectivity 
to $\mathrm{Ker} (v_{n+1}) $ in odd degrees, as 
required.
\end{proof}

For $X$ a spectrum and $n \in \nat$, the reduction map $\rho^n_{n-1}$ fits into
a diagram 
\[
 \xymatrix{
 \vtor
 \ar@{^(->}[r]
 &
 \bpn^* (X)
 \ar[d]^{\rho^n_{n-1}}
 \\
 \vtor[n-1]
 \ar@{^(->}[r]
 &
 \bpn[n-1]^* (X).
 }
\]
It is tempting to assert that the diagram can be completed to a commutative
square, using the structure 
theory of $BP_*BP$-comodules \cite[Theorem 0.1]{JY} and the stable comodule
structure on $BP^* (X)$ provided (after suitable completion)
by \cite[Sections 11, 15]{Board}. However, the passage to the Wilson theories
$\bpn$ is delicate.

For this reason, the hypothesis that $\rho^n_{n-1}$ sends $\vtor$ to
$\vtor[n-1]$ is 
included in the following result.

\begin{prop}
\label{prop:BP_surjectivity}
 Let $X$ be a spectrum and $n \in \nat$ such that 
\begin{enumerate}
 \item 
$\tilde{\sigma}_n : BP^* (X) \oplus \vtor \twoheadrightarrow
 \bpn^* (X)$ is surjective; 
\item
for $0 \leq j \leq n$
\begin{enumerate}
\item
$\vtor[j]\hookrightarrow \bpn[j]^* (X) \rightarrow \bpn[j-1]^* (X)$ factors
across $\vtor[j-1] \subset \bpn[j-1]^* (X)$;
\item
${\sigma}_j : \bpn[j]^* (X) \oplus \vtor[j-1] \twoheadrightarrow
 \bpn[j-1]^* (X)$ is surjective.
\end{enumerate}
\end{enumerate}
Then, for $0 \leq j \leq n $, $\tilde{\sigma}_j : BP^* (X) \oplus
\vtor[j] \twoheadrightarrow
 \bpn[j]^* (X)$ is surjective;
\end{prop}
 
\begin{proof}
The result is proved by a straightforward downward induction on $j$.
\end{proof}

\begin{rem}
 The result will be applied in the case where $\rho_n : BP^* (X) \rightarrow
\bpn^* (X)$ is itself surjective, hence establishing the first point of the
hypotheses.
\end{rem}

\subsection{Injectivity and base change}

The reduction map $\rho^{n+1}_n$ induces a morphism of $\bpn^*$-modules:
\[
 \bpn^* \otimes _{\bpn[n+1]^*}\bpn[n+1]^* (X)
 \rightarrow 
 \bpn^* (X)
\]
and, by base change,
\[
 \field_p \otimes_{\bpn[n+1]^*} \bpn[n+1]^* (X)
 \rightarrow 
 \field_p\otimes _{\bpn^*} \bpn^* (X)
\]
which need not {\em a priori} be injective. Criteria for the injectivity of this
and related morphisms are considered in this 
section.

The following terminology is used:
 
 \begin{defn}
 A $\bpn^*$-module $M$ is trivial if it is given by restriction of a
$\field_p$-vector space structure along $\bpn^* \twoheadrightarrow \field_p$.
\end{defn}

The following basic lemma extracts the formal part of the argument employed in
Propositions \ref{prop:quotient_injectivity} and
\ref{prop:BP_quotient_injectivity} below.

\begin{lem}
 \label{lem:hom-alg}
Let $\mathscr{C}$, $\mathscr{D}$ be abelian categories and 
\[
 \xymatrix{
A \ar[r]^\alpha  
\ar[d]
&
B
\ar[d]
\\
C
\ar[r]
_\gamma
&
D
}
\]
be a cartesian square in $\mathscr{C}$. Then 
\begin{enumerate}
 \item 
$\alpha$ is injective if and only if $\gamma$ is injective;
\item
the square is also cocartesian if and only if the associated total complex 
\[
 A \rightarrow B \oplus C \rightarrow D
\]
is a short exact sequence. 
\end{enumerate}
Suppose that the square is both cartesian and cocartesian and $F : \mathscr{C}
\rightarrow \mathscr{D}$ is a right exact functor. If $F (\alpha)$ is
a monomorphism then 
\[
 \xymatrix{
F(A) \ar[r]^{F(\alpha)}  
\ar[d]
&
F(B)
\ar[d]
\\
F(C)
\ar[r]
_{F(\gamma)}
&
F(D)
}
\]
is both cartesian and cocartesian and $F(\gamma)$ is injective.
\end{lem}

\begin{prop}
 \label{prop:quotient_injectivity}
 For $X$  a  spectrum and $n \in \nat$ such that 
 \begin{enumerate}
  \item 
  $\vtor$ is trivial as a $\bpn^*$-module;
  \item
  $\sigma_n : \bpn[n+1]^* (X) \oplus \vtor \twoheadrightarrow \bpn^* (X)$ is
surjective;
 \end{enumerate}
the morphism induced by $\rho^{n+1}_n$:
\[
 \field_p \otimes_{\bpn[n+1]^*} \bpn[n+1]^* (X)
 \rightarrow 
 \field_p\otimes _{\bpn^*} \bpn^* (X)
\]
is injective.
\end{prop}

\begin{proof}
 The surjection $\sigma_n$ induces a short exact sequence
 \[
  0 
  \rightarrow 
  K_n 
  \rightarrow 
  (\bpn[n+1]^*(X) /v_{n+1})
  \oplus 
  \vtor
  \twoheadrightarrow 
  \bpn^* (X)
  \rightarrow
  0
 \]
of $\bpn^*$-modules, corresponding to a cartesian and cocartesian diagram (of
monomorphisms)
\[
 \xymatrix{
 K_n
 \ar@{^(->}[r]
 \ar@{^(->}[d]
 &
 \vtor
  \ar@{^(->}[d]
  \\
  \bpn[n+1]^*(X) /v_{n+1}
   \ar@{^(->}[r]
   &
   \bpn^*(X).
 }
\]
Lemma \ref{lem:hom-alg} applied to the right exact functor $\field_p
\otimes_{\bpn^*} -$ implies that
 $\field_p \otimes_{\bpn[n+1]^*}\bpn[n+1]^*(X) \rightarrow \field_p
\otimes_{\bpn^*}\bpn^* (X)$ is  injective, since $\vtor$ is a trivial
$\bpn^*$-module, so that $\field_p\otimes_{\bpn^*} (K_n
\hookrightarrow \vtor)$ identifies with the monomorphism $K_n \rightarrow
\vtor$.
\end{proof}

The method of proof can also be applied to consider the morphism
\[
 \bpn^* \otimes _{BP^*} BP^* (X) \hookrightarrow \bpn^* (X)
\]
 induced by $\rho_n : BP \rightarrow \bpn$.

\begin{prop}
 \label{prop:BP_quotient_injectivity}
 For $X$  a  spectrum and $n \in \nat$ such that 
 \begin{enumerate}
\item 
$\bpn^* \otimes _{BP^*} BP^* (X) \hookrightarrow \bpn^* (X)$ is injective; 
  \item 
 $\vtor = \mathrm{Ker} (v_n)$;
  \item
  $\tilde{\sigma}_n : BP^* (X) \oplus \vtor \twoheadrightarrow \bpn^* (X)$ is
surjective;
 \end{enumerate}
the  morphism induced by $\rho_{n-1}$:
\[
 \bpn[n-1]^* \otimes_{BP^*} BP^* (X)
 \rightarrow 
\bpn[n-1]^* (X)
\]
is injective.
\end{prop}

\begin{proof}
 The hypotheses provide a short exact sequence 
\[
 0 
\rightarrow 
L_n 
\rightarrow 
(\bpn^* \otimes _{BP^*} BP^* (X)) \oplus \vtor
\rightarrow
\bpn^* (X)
\rightarrow 
0
\]
such that the components $L_n \rightarrow \vtor$ and $L_n \rightarrow \bpn^*
\otimes _{BP^*} BP^* (X)$ are injective. As in the proof of Proposition
\ref{prop:quotient_injectivity}, applying $\bpn[n-1]^* \otimes_{\bpn^*} -$
(which identifies with  $\zed_{(p)} \otimes_{\zed_{(p)} [v_n] } - $) yields the
horizontal short exact sequence below:
\begin{eqnarray}
\label{eqn:Ln_ses}
\noindent
\\
\nonumber
 \xymatrix{
L_n 
\ar[r]
&
\bpn[n-1]^*\hspace{-3pt}\otimes _{BP^*}\hspace{-3pt} BP^* (X) \oplus \vtor 
\ar[r]
&
\bpn[n-1]^* \hspace{-3pt}\otimes _{\bpn^*}\hspace{-3pt} \bpn^* (X)
\ar@{^(.>}[d]
\\
&
&
\bpn[n-1]^* (X),
} 
\end{eqnarray}
where the additional vertical inclusion is induced by $\rho_{n-1}^n$. The  injectivity of
the left hand horizontal morphism follows from the fact that multiplication by
$v_n$ acts trivially on
$\vtor$, so that   $\zed_{(p)}
\otimes_{\zed_{(p)} [v_n] } ( L_n \hookrightarrow \vtor)$ identifies with the
inclusion $L_n \hookrightarrow \vtor$.

By Lemma \ref{lem:hom-alg}, it follows  that $$\bpn[n-1]^* \otimes _{BP^*} BP^* (X) \rightarrow
\bpn[n-1]^* \otimes _{\bpn^*} \bpn^* (X)$$
is injective.  Composing with the vertical
monomorphism completes the proof.
\end{proof}

\section{Controlling the $v_n$-torsion in odd degrees}
\label{sect:noeth}

\subsection{The bounded torsion case}

The following result shows how the $v_n$-torsion can be understood in odd
degrees under suitable hypotheses. 
This will be applied in Section \ref{sect:trivial_torsion_thm} to deduce the
main general result of the paper, Theorem \ref{thm:general}.

\begin{prop}
\label{prop:bounded_torsion_trivial}
 Let $X$ be a spectrum and  $n \in \nat$ be  such that
$H\zed_{(p)}^\mathrm{odd}(X)=\bpn[0] ^\mathrm{odd} (X) \hookrightarrow H\field_p
^\mathrm{odd}(X)$ and there exists $N \in \nat$ such that, for $0 \leq j \leq
n$:
\begin{enumerate}
\item
$v_j^{N} \bpn[j]^\mathrm{odd} (X) =0$;
 \item 
 $\mathrm{image}(\rho^j_{-1})^{\mathrm{odd}} =  \mathrm{Im}
(Q_0 \ldots Q_j) ^\mathrm{odd} \subset H\field_p ^\mathrm{odd}(X) $;
\end{enumerate}
then, for $0 \leq j \leq n$, $\rho^j_{-1}$ induces an isomorphism:
\[
 \bpn[j]^\mathrm{odd} (X) \cong \mathrm{Im} (Q_0\ldots Q_j)^\mathrm{odd}.
\]
In particular, $\bpn[j] ^\mathrm{odd} (X) $ is a trivial $\bpn^*$-module.
\end{prop}
 
 \begin{proof}
The result is proved by upward induction upon $j$, starting with
$j=0$, which forms part of the hypothesis.
For the inductive step, suppose the result established for smaller values of
$j$. We require to prove that 
the reduction $\bpn[j]^* (X)\rightarrow H\field_p^* (X)$ is injective in odd
degrees. 

By the inductive hypothesis, the kernel of  ${\rho^j_{-1}} :
\bpn[j]^\mathrm{odd}
(X) {\rightarrow} H\field_p^\mathrm{odd} (X)$ coincides with the kernel 
of  $\rho^j_{j-1} : \bpn[j]^\mathrm{odd} (X) {\rightarrow} \bpn[j-1]
^\mathrm{odd}
(X)$, which is the image of multiplication by $v_j$ (restricted to odd degrees).
Hence, if $x=x_0$ (of odd degree)  is in the kernel of
$\rho^j_{-1}$, there is an odd degree element 
$x'_1$ such that $v_j x'_1= x_0$. Now $\rho^j_{-1}(x'_1)= Q_0 \ldots Q_j
y_1$ for some $y_1
\in H\field_p^* (X)$, by
the hypothesis on the image of $(\rho^j_{-1})^\mathrm{odd}$. Thus, consider the
element $x_1:= x'_1 - (\pm)(q_j\ldots q_0)y_1$, where the sign is taken so that
$\rho^j_{-1}(\pm(q_j\ldots q_0)y_1)=  Q_0 \ldots Q_j y_1$ (using Lemma
\ref{lem:Q_compatible});
 by construction $\rho^j_{-1}  (x_1) =0$ and $v_j x_1= v_j x'_1 = x_0$.

 Suppose  $x_0 \neq 0$, then  $ x_1 \neq 0$ and the argument can be repeated 
to form a sequence of non-trivial elements $x_s \in \bpn[j] ^\mathrm{odd} (X)$
 such that $v_j ^s x_s =x_0\neq 0$, $s\in \nat$. This contradicts the hypothesis
that
the $\bpn[j]^{\mathrm{odd}} (X)$ is bounded $v_j$-torsion; hence 
$(\rho^j_{-1})^\mathrm{odd}$ is injective.
 \end{proof}

\begin{rem}
\label{rem:Q_0_hypotheses}
\ 
 \begin{enumerate}
  \item 
The hypothesis on the image of $(\rho^j _{-1}) ^\mathrm{odd}$ is implied, for
example,  by the condition
$
 \hq (X, j)^\mathrm{odd} =0.
$ 
\item
If $H\field_p^* (X)$ is $Q_0$-acyclic, then $\bpn[0]^* (X) = H\zed_{(p)}^* (X)$
embeds in 
$H\field_p^*(X)$; in particular, the required embedding hypothesis holds in odd
degrees.
 \end{enumerate}
\end{rem}

\subsection{The Noetherian case}

When $X$ is a space, the bounded torsion hypothesis required in Proposition
\ref{prop:bounded_torsion_trivial}
can sometimes be provided by exploiting the algebra structure of $\bpn^* (X_+)$,
in particular in the presence of a finiteness hypothesis.

\begin{prop}
\label{prop:noeth}
 Let $X$ be a space and $n \in \nat$ such that $BP^\mathrm{odd} (X)=0$, 
$\bpn^* (X_+)$ is a Noetherian algebra and $\bpn^\mathrm{odd} (X)=0$. 

Then, for all $j \leq n$, $\bpn[j]^\mathrm{odd} (X)$ is a Noetherian $\bpn^*
(X_+)$-module and there exists $N \in \nat$ such that 
\[
 v_i^{N} \bpn[j]^\mathrm{odd} (X) =0.
\]
for all $0 \leq i \leq j$.
\end{prop}

\begin{proof}
The fact that $\bpn[j]^\mathrm{odd}(X)$ is a Noetherian $\bpn^* (X_+)$-module is
proved by a standard downward induction 
upon $j$. Proposition \ref{prop:BPodd0_v_n-torsion} implies that, for $0 \leq i \leq j$,
$\bpn[j]^\mathrm{odd}(X)$ is  $v_i$-torsion. Since, for each $j$, 
$\bpn[j]^\mathrm{odd} (X)$ is finitely-generated over
$\bpn^* (X_+)$, there is a uniform bound on the torsion.
\end{proof}

\section{Criteria for trivial torsion}
\label{sect:trivial_torsion_thm}

The following is the main general result of the paper; it is applied in the 
following section to the case $X =BV$, for $V$ an elementary abelian $p$-group.

\begin{thm}
\label{thm:general}
 Let $X$ be a space and $n \in \nat$ for which the following hypotheses are
satisfied:
\begin{enumerate}
 \item 
$BP^\mathrm{odd}(X)= \bpn^\mathrm{odd} (X) =0$;
\item
$\bpn^* (X_+)$ is Noetherian;
\item
$\bpn[0]^* (X) \hookrightarrow H\field_p^* (X)$ is a monomorphism with image
$\mathrm{Im}(Q_0)$;
\item
$\hq (X,j)^\mathrm{odd}=0$ for $0 \leq j \leq  n$.
\end{enumerate}
Then, for $0 \leq j  \leq  n$, $\vtor[j]$ is a trivial $\bpn[j]^*$-module which
identifies as:
\begin{eqnarray*}
\vtor[j] & \cong &\mathrm{Im} (q_j \ldots
q_0) \subset \bpn[j]^* (X)
\\
&\cong &
\mathrm{Im} (Q_0 \ldots
Q_j) \subset H\field_p^* (X);
\end{eqnarray*}
 in particular $\bpn[j]^\mathrm{odd} (X) \cong \mathrm{Im}(Q_0\ldots
Q_j)^\mathrm{odd}$.

Moreover:
\begin{enumerate}
\item
the reduction map $\rho^{j}_{j-1}$ induces a monomorphism
$\vtor[j] \hookrightarrow \vtor[j-1]$, 
which corresponds to the natural inclusion 
$\mathrm{Im} (Q_0 \ldots Q_{j} ) \hookrightarrow \mathrm{Im} (Q_0\ldots
Q_{j-1})$;
\item
$\sigma_j : \bpn[j+1]^* (X) \oplus \vtor[j] \rightarrow \bpn[j]^* (X)$ is
surjective;
\item
the reduction map $\rho^j_{-1}$ induces a monomorphism 
\[
 \field_p \otimes _{\bpn[j]^*} \bpn[j]^* (X)
\hookrightarrow H\field_p^* (X).
\]
\end{enumerate}
If, furthermore, $\bpn[j]^* (X) \twoheadrightarrow \hq^* (X,j)$ is surjective
for
$0 \leq j \leq n$, then 
 the reduction map $\rho^j_{-1}$ induces an isomorphism 
\[
 \field_p \otimes _{\bpn[j]^*} \bpn[j]^* (X) \cong \bigcap _{i=0}^j
\mathrm{Ker}(Q_i) \subset H\field_p^* (X). 
\]
\end{thm}

\begin{proof}
 Under the given hypotheses, by Remark \ref{rem:Q_0_hypotheses}, Propositions
\ref{prop:bounded_torsion_trivial} and \ref{prop:noeth} together 
 apply to determine $\bpn[j]^\mathrm{odd} (X)$ for $0 \leq j \leq  n$. 

Consider $\mathrm{Ker} (v_j)$ in  degree $2t + |v_j|$, for $j \geq 1$; this
fits into a commutative diagram:
\[
 \xymatrix{
\bpn[j]^{2t-1} (X) 
\ar@{^(->}[r] 
&
\bpn[j-1]^{2t-1}(X)
\ar@{->>}[r]
\ar@{^(->}[d]
\ar@{.>}[rd]_{\alpha}
&
\mathrm{Ker}(v_j)^{2t+|v_j|}
\ar@{^(->}[rd]
\ar[d]^\nu
\\
&
H\field_p^{2t-1} (X)
\ar[r]_{\pm Q_j} 
&
H\field_p^{2t+|v_j|}(X)
&
\bpn[j]^{2t+|v_j|}(X),
\ar[l]^{\rho^j_{-1}}
}
\]
where the top row is the short exact sequence (\ref{eqn:ses_fund}) and the
commutative square is provided by 
Lemma \ref{lem:Q_compatible}. Here, by the odd degree case, the morphism 
$\bpn[j]^{2t-1} (X) \rightarrow
\bpn[j-1]^{2t-1} (X)$ identifies as the  monomorphism   
$$\mathrm{Im} (Q_0 \ldots Q_j) ^{2t-1} \hookrightarrow \mathrm{Im}(Q_0 \ldots
Q_{j-1})^{2t-1}.$$

The morphism $\alpha$ indicated by the dotted arrow factors as 
\[
 \mathrm{Im} (Q_0 \ldots Q_{j-1}) ^{2t-1} 
\stackrel{\pm Q_j}{\twoheadrightarrow}
 \mathrm{Im} (Q_0 \ldots Q_{j}) ^{2t+|v_j|} 
\subset H\field_p^{2t+|v_j|} (X), 
\]
hence has kernel $\big(\mathrm{Ker} (Q_j) \cap \mathrm{Im}(Q_0\ldots Q_{j-1})
\big) ^{2t-1}$, which contains 
$\mathrm{Im} (Q_0 \ldots Q_j ) ^{2t-1}$. 
 The quotient $\big(\mathrm{Ker} (Q_j) \cap \mathrm{Im}(Q_0\ldots Q_{j-1})
/ \{
\mathrm{Im} (Q_0 \ldots Q_j )
\}
 \big) ^{2t-1}$ embeds in $\hq (X, j)^ \mathrm{2t-1}$ hence is trivial, by
hypothesis. Thus the kernel of $\alpha$ coincides with the image of
$\bpn[j]^{2t-1} (X)$ in $\bpn[j-1]^{2t-1} (X)$.  It follows that the vertical
morphism $\nu$ induces an isomorphism:
\[
 \mathrm{Ker}(v_j)^{2t+|v_j|} 
\stackrel{\cong}{\rightarrow}
 \mathrm{Im} (Q_0 \ldots Q_{j}) ^{2t+|v_j|} 
\subset H\field_p^{2t+|v_j|} (X).
\]
In particular,  the composite $\mathrm{Ker} (v_j)^{2t+|v_{j}|} \subset
\bpn[j]^{2t
+|v_{j}|} (X) 
\rightarrow H\field_p ^{2t+|v_{j}|}(X)$ is a monomorphism. Hence, by Lemma
\ref{lem:ker-tor}, in even degrees
\[
 \mathrm{Ker}(v_{j}) ^{2*} = \vtor[j]^{2*}\cong \mathrm{Im} (Q_0 \ldots Q_{j})
^{2*}.
\]
This completes the proof of the main statement.

If $j \geq 0$, since $\vtor[j]$ maps injectively to $H\field_p^* (X)$ by
$\rho^{j}_{-1}$, which factorizes as $\rho^{j-1}_{-1}
\rho^{j}_{j-1}$, it is clear that $\vtor[j] \hookrightarrow \vtor[j-1]$ is
injective and is as stated. Moreover,  the morphism $\kappa_j : \vtor[j]^*
\rightarrow
\mathrm{Ker} (v_{j+1}) ^{*+|Q_{j+1}|} $ is the surjection 
\[
 \mathrm{Im} (Q_0 \ldots Q_j) ^* 
\twoheadrightarrow 
\mathrm{Im} (Q_0 \ldots Q_{j+1})^{*+|Q_{j+1}|} 
\]
induced by $\pm Q_{j+1}$. Thus, by Lemma \ref{lem:equivalent_surjectivity}, the
morphism $\sigma_{j}$ is surjective.

This allows Proposition \ref{prop:quotient_injectivity} to be applied for $0
\leq j\leq n$ to deduce, by increasing  induction on $j$,
that the reduction morphism  $\rho^j_{-1}$ induces a monomorphism
\[
 \field_p \otimes _{\bpn[j]^*} \bpn[j]^* (X)
\hookrightarrow 
H\field_p^* (X).
\]

Finally, under the additional hypothesis of surjectivity to $\hq^*(X, j)$, the
image is identified by Corollary \ref{cor:surjectivity}.
\end{proof}

\begin{cor}
\label{cor:surjectivity_sigma_tilde}
 Under the hypotheses of Theorem \ref{thm:general}, if, in addition, the
reduction morphism 
\[
 \rho_n : BP^*(X) \twoheadrightarrow 
\bpn^* (X)
\]
is surjective, then, for each $0\leq j \leq n$, the morphisms $\rho_j :
BP\rightarrow \bpn[j]$ and $
q_j \ldots q_0 : H\field _p \rightarrow \Sigma^{\sum|Q_i|} \bpn[j]$ induce a
surjection
\[
 BP^* (X) \oplus H\field_p^{*-\sum|Q_i|}(X) \twoheadrightarrow 
\bpn[j]^* (X) .
\]
\end{cor}

\begin{proof}
The result follows by combining the conclusions of Theorem  \ref{thm:general}
with Proposition \ref{prop:BP_surjectivity}.
\end{proof}

This Corollary can be strengthened under an additional hypothesis. The statement
of the following result uses the conclusions of 
Proposition \ref{prop:BP_surjectivity}, Theorem \ref{thm:general} and the
notation introduced in Proposition \ref{prop:BP_quotient_injectivity}.

\begin{prop}
 \label{prop:BP-base-change}
Under the hypotheses of Theorem \ref{thm:general} if, in addition, the reduction
morphism $\rho_n$ induces an isomorphism
\[
\bpn^* \otimes_{BP^*}  BP^*(X) \stackrel{\cong}{\rightarrow} 
\bpn^* (X)
\]
then, for $0 \leq j \leq n$, the morphism $\rho_j$ induces a monomorphism 
\[
\bpn[j]^* \otimes_{BP^*}  BP^*(X) \hookrightarrow 
\bpn[j]^* (X)
\]
and, if $L_j$ denotes the kernel of the surjection 
\[
 (\bpn[j]^* \otimes_{BP^*} BP^* (X)) 
\oplus 
\vtor[j] 
\twoheadrightarrow 
\bpn[j]^* (X), 
\]
then $L_n= \vtor$ and, for $n \geq j\geq 0$, the inclusion $\vtor[j]
\hookrightarrow \vtor[j-1]$ induces a short
exact sequence
\[
 0\rightarrow L_j \rightarrow L_{j-1} \rightarrow \big( \mathrm{Ker}(Q_j) \cap
\mathrm{Im} (Q_0 \ldots Q_{j-1})\big)/ \mathrm{Im} (Q_0 \ldots  Q_j )
\rightarrow 0.
\]
\end{prop}

\begin{proof}
The injectivity of $\bpn[j]^* \otimes_{BP^*}  BP^*(X) \hookrightarrow \bpn[j]^*
(X)$ follows by applying Proposition \ref{prop:BP_quotient_injectivity}, using
the conclusions of Theorem \ref{thm:general}.  
The proof  of the remaining statements extends the methods of the proof of
Proposition \ref{prop:BP_quotient_injectivity}, using the fact that the
reduction $\rho^j_{j-1}$ induces 
a monomorphism $\iota_j : \vtor[j] \hookrightarrow \vtor[j-1]$. 

Since $\bpn[j]^* \otimes_{BP^*}  BP^*(X)$ is concentrated in even degrees and
$\vtor[j]^{\mathrm{odd}}$ coincides with $\bpn[j]^
{\mathrm{odd}} (X)$, $L_j$ is concentrated in even degrees. Moreover, since
$L_j$ injects to
$\vtor[j]$, $L_j$ is a trivial $\bpn[j]^*$-module. 

It is straightforward to show  that $L_n=\vtor$. Then, for $n \geq j >0$, the
diagram (\ref{eqn:Ln_ses}) of the proof of Proposition
\ref{prop:BP_quotient_injectivity} extends to a commutative diagram in which the
rows and columns are short exact sequences:
\[
 \xymatrix @C=1pc
{
L_j
\ar[r]
\ar[d]
&
\bpn[j-1]^* \hspace{-3pt} \otimes _{BP^*} \hspace{-3pt}BP^* (X)\oplus \vtor[j] 
\ar[r]
\ar[d]_{1\oplus \iota_j}
&
\bpn[j-1]^* \hspace{-3pt}\otimes _{\bpn[j]^*} \hspace{-3pt}\bpn[j]^* (X)
\ar@{^(->}[d]
\\
L_{j-1}
\ar[d]
\ar[r]
&
\bpn[j-1]^*\hspace{-3pt}\otimes _{BP^*}\hspace{-3pt} BP^* (X) \oplus \vtor[j-1] 
\ar[r]
\ar[d]
&
\bpn[j-1]^* (X)
\ar[d]\\
L_{j-1}/L_{j}
\ar[r]
&
\vtor[j-1] / \vtor[j]
\ar[r]
&
\mathrm{Ker}(v_j)^{*+|Q_j|}
}
\]
and the right hand column is the universal coefficient short exact sequence. 
This diagram
provides the natural inclusion of $L_j$ to $L_{j-1}$. (As
remarked above, the case of interest is where the degree of the middle column is
even; the odd degree case has 
already been used in the proof of Theorem \ref{thm:general}.)

Theorem \ref{thm:general} identifies the quotient $\vtor[j-1]/\vtor[j]$ as
$\mathrm{Im} ( Q_0 \ldots Q_{j-1} ) / \mathrm{Im} (Q_{0} \ldots Q_j)$ and
$\mathrm{Ker}(v_j)^{*+|Q_j|}$ as $\mathrm{Im} (Q_0 \ldots Q_j)$, in
appropriately shifted degree; the surjection is induced by the Milnor derivation
$Q_j$. The identification of the subquotient $L_{j-1}/L_j$  follows.
\end{proof}

\section{The case of elementary abelian $p$-groups}
\label{sect:eltab}

\subsection{Generalized Margolis vanishing}

The structure of $H\field_p^*( BV_+)$ is well-known; by the Künneth theorem, it 
suffices to describe the rank one case. For
$p$ odd, $H\field_p^* (B\zed/p_+) \cong \Lambda (u) \otimes \field_p [v]$,
with $|u|=1$ and $|v|=2$ with Bockstein $\beta u = v$; for $p=2$,  $H\field_2^*
(B\zed/2_+) \cong  \field_2 [u]$, where $|u|=1$. The action of the Milnor
primitives is determined as follows: for $p$ odd  $Q_i$ acts trivially on $v$ 
and $Q_i u = v^{p^i}$;  for $p=2$, $Q_i u =
u^{2^{i+1}}$.

For $p$ odd and $V$ an elementary abelian $p$-group of finite rank, the above
gives the  
isomorphism 
\[
 H\field_p^* (B V_+) \cong \Lambda^* (V^\sharp) \otimes S^*(V^\sharp),
\]
where $V^\sharp$ denotes the linear dual, $\Lambda^*$ the exterior algebra
 and $S^*$ the symmetric algebra.  This provides a bigrading which is
related to 
the standard grading by $H\field_p^n (B V_+) \cong \bigoplus _{a+2b=n} \Lambda^a
(V^\sharp) \otimes S^b(V^\sharp)$. The Milnor primitives respect the
decomposition, in the sense that 
\begin{eqnarray}\label{eqn:Milnor_bigrading}
Q_i : \Lambda ^a (V^\sharp) \otimes S^b (V^\sharp) \rightarrow  \Lambda ^{a-1}
(V^\sharp) \otimes S^{b+p^i} (V^\sharp).
\end{eqnarray}

This is a Koszul-complex type differential.

\begin{rem}
\label{rem:filter_F[u]}
Similar statements are obtained for $p=2$ by filtering, based 
on the isomorphism of $\field_2 [u^2]$-modules: $\field_2 [u] \cong \Lambda (u)
\otimes \field_2 [u^2]$.
\end{rem}

The map  $B \zed/p \rightarrow \mathbb{C}P^\infty$ induced by the inclusion 
$\zed/p \subset S^1$ of $p$th roots of unity, induces a morphism of
unstable algebras $H\field_p^*
(\mathbb{C}P^\infty_{\ +}) \cong \field_p [x] \hookrightarrow
H\field_p^*(B\zed/p_+)$, with $|x|=2$, determined  by $x \mapsto v$
(respectively $x \mapsto u^2$ for $p=2$). Since $ H\field_p^* 
((\mathbb{C}P^\infty) ^{\times d}_{\ +}) $ is concentrated in even degrees, 
as observed in Remark \ref{rem:free_trivial_hq}, it can be identified with $\hq 
((\mathbb{C}P^\infty) ^{\times d}_{\ +})$.

\begin{thm}
\label{thm:multikoszul}
 Let $V$ be an elementary abelian $p$-group of rank $d$ and $n \in \nat$. Then
the map $B \zed/p \rightarrow \mathbb{C}P^\infty$ induces a surjection 
 \[
  H\field_p^* ((\mathbb{C}P^\infty) ^{\times d}_{\ +}) 
  = 
  \hq ((\mathbb{C}P^\infty) ^{\times d}_{\ +})
  \twoheadrightarrow 
  \hq ^*(BV_+ , n).
 \]
 In particular $\hq ^* (BV_+, n)^{\mathrm{odd}} =0$ and the Thom reduction $BP
\rightarrow H\field_p$ induces a surjection 
 \[
  BP^* (BV_+) 
  \twoheadrightarrow  \hq^* (BV_+ , n).
 \]
\end{thm}

\begin{proof}
 The case $p$ odd is treated below; the argument adapts to the case $p=2$ by
filtering  using the number of terms of odd degree in monomials (cf. Remark
\ref{rem:filter_F[u]}).
 
 Since $H\field_p^* (\mathbb{C}P^\infty_+)$ is concentrated in even degrees,
the Milnor operations act trivially, hence the Thom reduction maps to 
 $\bigcap_{i=0}^{\infty} \mathrm{Ker} (Q_i)$ and the morphism to $\hq^* (BV_+,
n)$ 
is defined. The first statement is proved using a refinement of Margolis' 
criterion for the
freeness of modules over exterior algebras 
 \cite[Theorem 8(a), Section 18.3]{margolis}, exploiting the filtration induced
by the number of exterior generators.

Namely, a straightforward reduction (using the behaviour 
(\ref{eqn:Milnor_bigrading}) of
the Milnor operations with respect to the bigrading)
implies that it is sufficient to show that an element $x \in \Lambda^a(W)
\otimes S^b(W)$ which lies in $\bigcap_{i=0}^{n} \mathrm{Ker} (Q_i)$
 is in the image of 
 \[
  (Q_0 \ldots Q_n) : \Lambda^{a+n+1}(W) \otimes S^{b- \sum |Q_i|}(W)
  \rightarrow 
  \Lambda^a(W) \otimes S^b(W),
 \]
where $W$ is written for $V^\sharp$, for notational simplicity. This is  a case
of Proposition \ref{prop:multi-koszul} below.

The Thom reduction $ BP^* ((\mathbb{C}P^\infty) ^{\times d}_{\
+})  \twoheadrightarrow H\field_p^* ((\mathbb{C}P^\infty) ^{\times d}_{\ +})$ is
surjective; moreover, Landweber showed  that $BP^*
((\mathbb{C}P^\infty) ^{\times d}_{\ +}) \twoheadrightarrow BP^* (BV_+) $ is
surjective (this is included in \cite[Proposition 2.3]{Str}). The final
statement follows. 
 \end{proof}

 The  acyclicity of the Koszul
complex is restated below; it is valid for all primes (with the appropriate 
interpretation of the 
 operation $Q_i$ at $p=2$).
 
 \begin{lem}
 \label{lem:Koszul_acyclic}
 For $i,a, n \in \nat$, the Koszul complex yields an acyclic complex:
 \[
  \ldots 
  \rightarrow 
  \Lambda^n \otimes S^{a- n p^i} 
  \stackrel{Q_i}{\rightarrow}
   \Lambda^{n-1} \otimes S^{a- (n-1)p^i} 
  \stackrel{Q_i}{\rightarrow}
 \ldots 
 \rightarrow 
 S^a 
 \rightarrow 
 \Strunc{a}{i}
 \rightarrow 
 0,
 \]
 where $\Strunc{a}{i}$ is the truncated symmetric power, imposing the relation
$w^{p^i}=0$.
\end{lem}
 
 The following result can be deduced from \cite[Theorem 8(a),
Section 18.3]{margolis}; a direct proof is given here, since this indicates the
very general nature of the result.
  
\begin{prop}
\label{prop:multi-koszul}
Let $W$ be an elementary abelian $p$-group of finite rank, $\emptyset \neq \cali
\subset \nat$ be a non-empty, finite indexing set and $0
< a \in \nat$.  If  $x \in \Lambda^a(W) \otimes S^b(W)$ is an element such that
$Q_i x = 0$, $\forall i \in \cali$, then there exists an element $y \in
\Lambda^{a+ |\cali|} (W) \otimes S^{b - \sum_{i \in \cali} |Q_i|}(W)$ such that 
\[
 x = (\prod_{i \in \cali} Q_i ) y.
\]
In particular, $x=0$ if either $a+ |\cali| > \dim W$ or $b< \sum_{i \in \cali} 
|Q_i|$.
\end{prop}

\begin{proof}
 The proof is by  induction on $|\cali|$, with an internal induction on
$|x|$, where the degree of an element of $\Lambda^a(W) \otimes S^b(W)$ is
$a+2b$. For $|\cali|=1$, the result holds by  Lemma \ref{lem:Koszul_acyclic};
the initial step of the $|x|$ induction is a straightforward consequence of
connectivity, since the degree of elements is non-negative. 

For the inductive step of the degree induction,  consider $x \in \Lambda^a(W)
\otimes S^b(W)$ and $\cali= \{ i_1 < i_2 <
\ldots < i_t \}$ as in the statement, supposing that the result holds for all
indexing sets $\mathscr{J}$ with $|\mathscr{J}|< |\cali|=t$ and for such
elements of degree  $< |x|$. 

To prove the result, it is sufficient to construct elements $\alpha_k \in
\Lambda^{a + t-1}(W) \otimes S^{b-\sum_{i \in
\cali\backslash\{i_1\} } |Q_i|} (W)$, for $t \geq k \geq 1$ which satisfy the
following properties
 \begin{eqnarray}
 \label{eqn:cond1}
(\prod_{i \in \cali\backslash\{i_1\} } Q_i ) \alpha_k &=& x\\
\label{eqn:cond2}
(\prod_{1 \leq s \leq k } Q_{i_s}) \alpha_k &=&0. 
 \end{eqnarray}

Indeed,  the element $\alpha_1$ then satisfies $Q_{i_1} \alpha_1=0$, so that
acyclicity of the complex of Lemma \ref{lem:Koszul_acyclic} implies the
existence of  $y$ such that $Q_{i_1} y =\alpha_1$. By condition
(\ref{eqn:cond1}) for $\alpha_1$, this satisfies  $ x = (\prod_{i \in \cali} Q_i
) y$, as required.
 
The construction of the $\alpha_k$ is by descending induction on $k$; by
induction upon $|\cali|$, there exists an element $\alpha_t$ such that $x=
(\prod_{i \in \cali\backslash\{i_1\} } Q_i ) \alpha_t$. Since $Q_{i_1} x=0$, by
hypothesis, the second condition required of $\alpha_t$ also holds, so this
forms the initial step of the descending induction.

For the inductive step ($t \geq k >1)$, consider $\alpha_k$ and form the element
$ \beta_k :=
(\prod_{1 \leq s < k } Q_{i_s}) \alpha_k$, which is the obstruction to
$\alpha_k$ being taken for $\alpha_{k-1}$. 
In the case $k=t$,  $|x|- |\beta_t|=|Q_{i_t}| - |Q_{i_1}|>0$.

Condition (\ref{eqn:cond2}) for  $\alpha_k$ implies that $Q_{i_s} \beta_k=0$, 
for
$1 \leq s \leq k$. Hence, the global inductive hypothesis in the proof of the
theorem yields an element $\gamma_k$
such that $(\prod_{1 \leq s \leq k} Q_{i_s} ) \gamma_k = \beta_k$ (for $k=t$,
this is induction on the degree, using the fact that $|\beta_t|<|x|$, and, for
$k<t$, induction on $|\cali|$).

Taking $\alpha_{k-1} := \alpha_k - Q_{i_k} \gamma_k$, the required conditions
are satisfied, completing the inductive step. 
\end{proof}

\begin{rem}
As suggested by the referee, it is interesting no observe the following consequence: 
  $ H\field_p ^* (BV_+) \cong 
\Lambda^*(V^\sharp) \otimes S^*(V^\sharp )$ contains no free $\Lambda (Q_i|i \in 
\cali)$-submodule  if $|\cali| > \dim V$. In the case $ |\cali| = \dim V$, $\Lambda^{ |\cali|} (V^\sharp)$ 
is one dimensional and the free $\Lambda (Q_i|i \in 
\cali)$-module  summand of  $ H\field_p ^* (BV_+)$ is generated by $\Lambda^{ |\cali|} (V^\sharp) \otimes S^* (V^\sharp)$.
 
 The reader may wish to compare Proposition \ref{prop:multi-koszul} with the analysis of the stable 
summands of $\Sigma^\infty BV_+$ which have cohomology that is free over 
$\Lambda (Q_i|i \in \cali)$; here  Margolis's criterion can be applied directly, as in \cite{CK}.
\end{rem}

\subsection{The structure of the $\bpn$-cohomology of elementary abelian
$p$-groups}

The description of $\bpn^* (BV)$ is obtained by applying Theorem
\ref{thm:general}.
 The required property of  the torsion of $\bpn^{\mathrm{odd}} (BV)$ is provided
by the
following:
 
 \begin{prop}
 \label{prop:noeth_strickland}
\cite[Proposition 2.3]{Str}
For $V$ an elementary abelian $p$-group of rank $d\leq n+1$, 
$
 \bpn^* (BV_+) 
$
is a Noetherian algebra concentrated in even degrees, which has no $p$-torsion
if $d < n+1$.
 \end{prop}

\begin{nota}
Write $\double H\field_p^* (BV) \subset
H\field_p^* (BV)$ for the augmentation ideal of the polynomial subalgebra if $p$
is odd and for the double $\Phi H\field_2^* (BV)$ if $p=2$. Thus, $\double H
\field_p^* (BV)$ coincides with the image of $\rho_{-1}$. 
\end{nota}

\begin{thm}
\label{thm:eltab}
 For $V$  an elementary abelian $p$-group of finite rank and $j \in \nat$, 
 the following statements hold:
\begin{enumerate}
 \item 
$\vtor[j]$ is a trivial $\bpn[j]^*$-module which identifies as:
 \begin{eqnarray*}
\vtor[j] & \cong &\mathrm{Im} (q_j \ldots
q_0) \subset \bpn[j]^* (BV)
\\
&\cong &
\mathrm{Im} (Q_0 \ldots
Q_j) \subset H\field_p^* (BV)
\end{eqnarray*}
and, in particular, $\bpn[j]^\mathrm{odd} (BV) \cong \mathrm{Im}(Q_0\ldots
Q_j)^\mathrm{odd}$.
\item
The reduction map $\rho^j_{-1}$ induces an isomorphism 
\[
 \field_p \otimes _{\bpn[j]^*} \bpn[j]^* (BV) \cong \bigcap _{i=0}^j
\mathrm{Ker}(Q_i) \subset H\field_p^* (BV). 
\]
\item
\label{item:reduction1}
 The reduction map $\rho_j$ induces a monomorphism 
  \[
   \bpn[j]^* \otimes_{BP^*} BP^* (BV) \hookrightarrow \bpn[j]^* (BV) 
  \]
which is an isomorphism modulo $v_j$-torsion and, is an isomorphism
for $j> \mathrm{rank}(V)$; in particular, 
$$\bpn[j] ^* (BV) \Big[\frac{1}{v_j}\Big] \cong
\bpn[j]^*\Big[\frac{1}{v_j}\Big]  \otimes _{BP^* } BP^* (BV).$$
\item
\label{item:reduction2}
The reduction map $\rho_j$ and localization induces a monomorphism 
\[
 \bpn[j]^* (BV) \hookrightarrow H\field_p^* (BV) \oplus 
\big(\bpn[j]^*\Big[\frac{1}{v_j}\Big]  \otimes
_{BP^*
} BP^* (BV) \big). 
\]
\item
The  morphism $\tilde{\sigma}_j$ induces a short exact sequence
\[
0\rightarrow
L_j \hookrightarrow  \big(\bpn[j]^* \otimes_{BP^*} BP^* (BV)\big) \oplus
\vtor[j] \twoheadrightarrow \bpn[j]^*
(BV) \rightarrow 0
\]
where $L_j$  is isomorphic to $\double H \field_p^* (BV) \cap \mathrm{Im} (Q_0
\ldots Q_j) \subset H\field_p^* (BV)$.
\end{enumerate}
\end{thm}

\begin{proof}
 The first two statements follow from  Theorem \ref{thm:general}, using Theorem
\ref{thm:multikoszul} and Proposition \ref{prop:noeth_strickland} to show that
the hypotheses are satisfied. Statements (\ref{item:reduction1}) and
(\ref{item:reduction2}) follow from Theorem \ref{thm:general} and Proposition
\ref{prop:BP-base-change}.

Finally, the identification of $L_j$ follows by analysing the information
furnished in Proposition \ref{prop:BP-base-change}, using the fact that the
image of $L_j$ in
$\vtor[j] \cong \mathrm{Im}(Q_0 \ldots Q_j) \subset H\field_p^* (BV)$ also lies
in 
$\double := \double H \field_p^* (BV)$. Namely, the proof is by downward 
induction on $j$, starting from $j>
\mathrm{rank}(V)$, for which the result is clear.  Theorem \ref{thm:multikoszul}
implies that 
$\double \hookrightarrow H \field_p^* (BV)$ induces a surjection $\double
\twoheadrightarrow \hq^* (BV, j)$, with kernel $\double \cap \mathrm{Im} (Q_0
\ldots  Q_j)$. 
The inductive step follows from the observation that the  square below is
cartesian:
\[
 \xymatrix{
 & \double \cap \mathrm{Im} (Q_0 \ldots  Q_{j-1})
 \ar@{^(->}[r]
 \ar@{->>}[d]
 &
 \double 
 \ar@{->>}[d]
 \\
 L_{j-1}/ L_j 
 \ar@{}[r]|(.25)\cong 
 &
 \big( \mathrm{Ker}(Q_j) \cap \mathrm{Im} (Q_0\ldots Q_{j-1})\big) /
\mathrm{Im}(Q_0 \ldots Q_j) 
  \ar@{^(->}[r]
  &
  \hq^* (BV, j),
 }
\]
where the isomorphism is given by Proposition \ref{prop:BP-base-change}.
\end{proof}

\begin{rem}
The method of proof applies {\em mutatis mutandis} to any spectrum 
constructed from $BP$ by forming the quotient by a cofinite subset of a 
suitable 
set of generators  $\{ v_i
| i \geq 0 \}$ of $BP_*$. 
\end{rem}

Theorem \ref{thm:eltab} yields the following precise description of the failure 
of 
surjectivity of the reduction map $\rho_j$ for the cohomology of $BV$, 
a far-reaching generalization of the result of Strickland \cite{Str}.

\begin{cor}
\label{cor:spectra_detection}
For $V$ an elementary abelian $p$-group of finite rank and $j\in \nat$, the
morphism 
$
 \xymatrix{
 BP \amalg \Sigma^{\sum |Q_i|} H \field_p 
 \ar[rr]^(.6){(\rho_j, q_j\ldots q_0)} 
 &
\  &
\bpn[j]
 }
$
induces a surjection 
\[
 BP^* (BV) \oplus H\field_p^{* -\sum |Q_i|}(BV)
 \twoheadrightarrow 
\bpn[j]^* (BV).
 \]
\end{cor}

\begin{proof}
Follows from Corollary \ref{cor:surjectivity_sigma_tilde}, using 
 \cite[Proposition 2.3]{Str} to treat the cases $n \gg 0$.
\end{proof}


\providecommand{\bysame}{\leavevmode\hbox to3em{\hrulefill}\thinspace}
\providecommand{\MR}{\relax\ifhmode\unskip\space\fi MR }
\providecommand{\MRhref}[2]{%
  \href{http://www.ams.org/mathscinet-getitem?mr=#1}{#2}
}
\providecommand{\href}[2]{#2}

\end{document}